\documentclass[11pt]{article}
\usepackage{amsmath,amssymb,amsthm,graphicx,subfigure,float,url}
\usepackage[colorlinks=true,citecolor=red,filecolor=red,urlcolor=magenta]{hyperref}
\usepackage{pdfsync}
\usepackage{stmaryrd}
\usepackage{mathrsfs}
\title{}

\author{}

\newcommand{\ep} {\varepsilon}
\newcommand{\om} {\omega}
\newcommand{\al} {\alpha}

\newcommand{\Om} {\Omega}
\newcommand{\la} {\lambda}

\newcommand{\N} {\mathbb{N}}

\newcommand{\C} {\mathbb{C}}
\newcommand{\R} {\mathbb{R}}

\renewcommand{\geq}{\geqslant}
\renewcommand{\leq}{\leqslant}
\newtheorem{theorem}{Theorem}  
\newtheorem{proposition}{Proposition}
\newtheorem{corollary}{Corollary}
\newtheorem{definition}{Definition}
\newtheorem{lemma}{Lemma}

\theoremstyle{definition}\newtheorem{remark}{Remark}
\newcommand{\Hun}{\mathbf{(H)}}

\newcommand{\Op}{\mathrm{Op}}

\date{}

\begin{document}
\title{Observability properties of the homogeneous wave equation on a closed manifold}
\author{Emmanuel Humbert\footnote{Laboratoire de Math\'ematiques et de Physique Th\'eorique, UFR Sciences et Technologie, Facult\'e Fran\c cois Rabelais, Parc de Grandmont, 37200 Tours, France (\texttt{emmanuel.humbert@lmpt.univ-tours.fr}).}
\and Yannick Privat\footnote{Sorbonne Universit\'es, CNRS UMR 7598, UPMC Univ Paris 06, Laboratoire Jacques-Louis Lions, F-75005, Paris, France (\texttt{yannick.privat@upmc.fr}).}
\and Emmanuel Tr\'elat\footnote{Sorbonne Universit\'es, UPMC Univ Paris 06, CNRS UMR 7598, Laboratoire Jacques-Louis Lions, F-75005, Paris, France (\texttt{emmanuel.trelat@upmc.fr}).}}
\maketitle

\begin{abstract}
We consider the wave equation on a closed Riemannian manifold. We observe the restriction of the solutions to a measurable subset $\omega$ along a time interval $[0, T]$ with $T>0$. It is well known  that, if $\omega$ is open and if the pair $(\omega,T)$ satisfies the Geometric Control Condition then an observability inequality is satisfied, comparing the total energy of solutions to their energy localized in $\omega \times (0, T)$. The observability constant $C_T(\om)$ is then defined as the infimum over the set of all nontrivial solutions of the wave equation of the ratio of localized energy of solutions over their total energy.  

In this paper, we provide estimates of the observability constant based on a low/high frequency splitting procedure allowing us to derive general geometric conditions guaranteeing that the wave equation is observable on a measurable subset $\omega$. 
We also establish that, as $T\rightarrow+\infty$, the ratio $C_T(\om)/T$ converges to the minimum of two quantities: the first one is of a spectral nature and involves the Laplacian eigenfunctions; the second one is of a geometric nature and involves the average time spent in $\omega$ by Riemannian geodesics.
\end{abstract}

 
\tableofcontents
\section{Introduction}
Let $(\Omega,g)$ be a compact connected Riemannian manifold of dimension $n$ without boundary. The canonical Riemannian volume on $\Om$ is denoted by $v_g$, inducing the canonical measure $dv_g$. Measurable sets are considered with respect to the measure $dv_g$.

Consider the wave equation
\begin{equation}\label{waveEqobs}
\partial_{tt}y- \triangle_g y=0\qquad \textrm{in }(0,T)\times \Omega
\end{equation}
where $\triangle_g$ stands for the usual Laplace-Beltrami operator on $\Omega$ for the metric $g$. Recall that the Sobolev space $H^1(\Omega)$ as the completion of the vector space of $C^\infty$ functions having a bounded gradient (for the Riemannian metric) in $L^2(\Omega)$ for the norm given by 
$\Vert u\Vert_{H^1}^2=\Vert u\Vert_{L^2}^2+\Vert \nabla u\Vert_{L^2}^2$ and that $H^{-1}(\Omega)$ is the dual space of $H^1(\Omega)$ with respect to the pivot space $L^2(\Omega)$. 

For every set of initial data $(y(0,\cdot),\partial_ty(0,\cdot))\in L^2(\Omega)\times H^{-1}(\Omega)$, there exists a unique solution $y\in \mathcal{C}^0(0,T;L^2(\Omega))\cap \mathcal{C}^1(0,T;H^{-1}(\Omega))$ of \eqref{waveEqobs}.

Let $T>0$ and let $\omega$ be an arbitrary measurable subset of $\Omega$ of positive measure. The notation $\chi_\omega$ stands for the characteristic function of $\omega$, in other words the function equal to $1$ on $\omega$ and $0$ elsewhere.
The \textbf{observability constant} in time $T$ associated to \eqref{waveEqobs} is defined by 
\begin{equation}\label{defCT}
C_T(\om) =  \inf \left\{ J_T^{\om}(y^0,y^1)  \ \mid\ (y^0,y^1) \in L^2(\Omega)\times H^{-1}(\Omega) \setminus\{(0,0)\} \right\}
\end{equation}
where 
\begin{equation}\label{defJTom}
J_T^{\om}(y^0,y^1) = \frac{\int_0^T\int_\om \vert y(t,x) \vert ^2 \, dv_g \, dt}{\Vert (y^0,y^1)\Vert_{L^2\times H^{-1}}^2}.
\end{equation}
In other words, $C_T(\om)$ is the largest possible nonnegative constant $C$ such that
\begin{equation*}
C \Vert (y^0,y^1)\Vert_{L^2\times H^{-1}}^2
\leq \int_0^T\int_\omega \vert y(t,x)\vert^2 \,dv_g(x) \, dt
\end{equation*}
for all $(y^0,y^1)\in L^2(\Omega)\times H^{-1}(\Omega)$ such that $(y(0,\cdot),\partial_ty(0,\cdot))=(y^0,y^1)$. 
The equation \eqref{waveEqobs} is said to be \textit{observable} on $\omega$ in time $T$ if $C_T(\om)>0$.
Note that, by conservation of energy, we always have $0\leq C_T(\om)\leq T$.
It is well known that if $\om$ is an open set then observability holds when the pair $(\omega,T)$ satisfies the \textit{Geometric Control Condition} in $\Omega$ (see \cite{BardosLebeauRauch,BurqGerard,rauch-taylor}), according to which every ray of geometric optics that propagates in $\Omega$ intersects $\omega$ within time $T$. 
This classical result will be slightly generalized to more general subsets $\omega$ within this paper. 
Let us also mention the recent article \cite{laurent-leautaud} where the authors provide sharp estimates of the observability constant at the minimal time at which unique continuation holds for the wave equation. 

This article is devoted to establishing various properties of the observability constant. 
Our main results are stated in Section \ref{sec:mainres}. We first show that, under appropriate assumptions on the observation domain $\omega$, the limit of $C_T(\om)/T$ as $T\to +\infty$ exists, is finite and is written as the minimum of two quantities: the first one is a {\it spectral quantity} involving the eigenfunctions of $-\triangle_g$ and the second one is a {\it geometric quantity} involving the geodesics of $\Om$. 
We then provide a characterization of observability (Corollary \ref{obser_carac})  based on a low/high frequency splitting  procedure (Theorem \ref{main_finite_time}) showing how observability can be characterized in terms of high-frequency eigenmodes. In turn, our approach gives a new proof of  
results of \cite{BardosLebeauRauch,rauch-taylor} on observability.
Finally, we investigate the case where there is a spectral gap assumption on the spectrum of $-\triangle_g$.

\section{Statement of the results}\label{sec:mainres}
Let $T >0$ and let $\omega$ be a measurable subset of $\Omega$.

Let $(\phi_j)_{j\in \N^*}$ be an arbitrary Hilbert basis of $L^2(\Omega)$ consisting of eigenfunctions of $-\triangle_g$, associated with the real eigenvalues $(\lambda_j^2)_{j\in\N^*}$ such that $0<\lambda_1\leq\lambda_2\leq\cdots\leq\lambda_j\rightarrow+\infty$.
For every $N \in\N$, we define 
\begin{multline}\label{CTgeqN1}
C_T^{> N}(\om)  = \inf  \{J_T^{\om}(y^0,y^1)\mid \langle y^i,\phi_j\rangle_{(H^i)',H^i}=0, \quad \forall i=0,1, \quad \forall j= 1,\ldots,N\\
(y^0,y^1) \in L^2(\Omega)\times H^{-1}(\Omega) \setminus\{(0,0)\} \}  
\end{multline} 
with the convention that $H^0=L^2$. Noting that $C_T(\om)\leq C_T^{> N}(\om)\leq C_T^{> N+1}(\om)$ for every $N\in\N$, we define the ``high-frequency'' observability constant as follows.

\begin{definition}[high-frequency observability constant]
The {\em high-frequency observability constant} $\alpha^T(\om)$ is defined by
 $$
 \alpha^T(\om) = \lim_{N  \to + \infty}  \frac{1}{T} C_T^{> N}(\om).
 $$
\end{definition}

This limit exists since the mapping $\N\ni N  \mapsto C_T^{> N}(\om)$ is nondecreasing and is bounded\footnote{This follows by conservation of the energy $[0,T]\ni t\mapsto \Vert \partial_t y(t,\cdot )\Vert^2_{L^2(\Omega)}+\Vert\nabla y(t,\cdot)\Vert^2_{L^2(\Omega)}$ for any solution $y$ of \eqref{waveEqobs}.\label{consNRJ}}. 

\begin{definition}[Spectral quantity $g_1(\om)$]
The \textbf{spectral quantity} $g_1(\om)$ is defined by 
\begin{equation*}
g_1(\om) = \inf_{\phi\in\mathcal{E}}  \frac{\int_{\om}  |\phi(x)|^2 \, dv_g}{\int_{\Omega} |\phi(x)|^2 \, dv_g},
\end{equation*}
where the infimum runs over the set $\mathcal{E}$ of all nonconstant eigenfunctions $\phi$ of  $-\triangle_g$. 
\end{definition}


\begin{center}
\fbox{\bf Main results on the observability constant $C_T(\om)$}
\end{center}

\begin{theorem}\label{main_finite_time}
We have
\begin{equation*}
 \frac{C_T(\om)}{T}  \leq \min \left( \frac{1}{2} g_1(\om), \alpha^T(\om) \right).
\end{equation*}
Moreover, if $\frac{C_T(\om)}{T} <  \alpha^T(\om)$ then the infimum in the definition of  $C_T$ is reached:  there exists $(y^0,y^1)\in L^2(\Omega)\times H^{-1}(\Omega) \setminus\{(0,0)\}$ such that 
$$
\frac{C_T(\om)}{T} = J_T^{\om}(y^0,y^1) > 0. 
$$ 
\end{theorem}

In what follows we are going to provide explicit estimates of $\alpha^T(\om)$, thus yielding observability properties. 

\begin{corollary}
 \label{obser_carac}
We have $C_T(\om) >0$ if and only if $\alpha^T(\om)>0$.
\end{corollary}

Note that this result is valid for any Lebesgue measurable subset $\omega$ of $\Omega$ and for any $T >0$. Corollary \ref{obser_carac} says that observability is a high-frequencies property, which was already known when inspecting the proofs of GCC in \cite{BardosLebeauRauch,LRLTT}, but the above equivalence with the notion of high-frequency observability constant, was never stated like that, up to our knowledge.
Besides, our objective is also to investigate what happens for measurable subsets $\om$ that are not open.

\begin{remark}
The results established in \cite{BardosLebeauRauch} are valid for manifolds having a nonempty boundary. Corollary \ref{obser_carac} above is still true in this context but extending the results hereafter to such geometries would require a deeper study of $\alpha^T(\om)$ on manifolds with boundary, which are beyond the scope of this paper 
\end{remark}

As a consequence of our techniques of proof, which are based on a concentration-compactness argument, we get the following large-time asymptotics of the observability constant $C_T(\om)$. 

\begin{theorem}[Large-time observability] \label{main_long_time}
The limit
$$
\alpha^\infty(\om)= \lim_{T \to +\infty} \alpha^T(\om)
$$ 
exists and we have
\begin{equation}\label{eq1thmain_long_time}
\lim_{T \to +\infty} \frac{C_T(\om)}{T} =  \min \left( \frac{1}{2} g_1(\om),\alpha^\infty(\om) \right) .
\end{equation}
Moreover, if $\frac{1}{2}g_1(\om) < \alpha^\infty(\om) $ then $g_1(\om)$ is reached, i.e., the infimum in the definition of $g_1(\om)$ is in fact a minimum.
\end{theorem}

Consequences of this result are given hereafter.

\begin{center}
\fbox{\bf Characterization of the quantities $\alpha^T(\om)$ and $\alpha^{\infty}(\om)$}
\end{center}

In what follows, we say that $\gamma$ is a \emph{ray} if $\gamma$ is the projection onto $\Omega$ of a Riemannian geodesic traveling at speed one in the co-sphere bundle of $\Omega$. We denote by $\Gamma$ the set of all rays of $\Omega$.

\begin{definition}[Geometric quantity $g_2(\om)$]\label{def:geomQuant}
We define
\begin{equation}
g_2^T(\om)= \inf_{\gamma\in \Gamma} \frac{1}{T} \int_0^T \chi_\omega(\gamma(t)) \, dt 
\end{equation}
and
\begin{equation}\label{def:g2}
g_2(\om) = \lim_{T \to +\infty}g_2^T(\om).
\end{equation}
\end{definition}

The quantity $g_2^T(\om)$ stands for the minimal average time spent by a geodesic $\gamma$ in $\omega$.
Note that the mapping $T\mapsto g_2^T(\om)$ is nonnegative, is bounded above by $1$ and is  subadditive. Hence the limit in the definition of $g_2(\om)$ is well defined.  

In \cite{HebrardHumbert}, it has been shown how to compute the geometric quantity $g_2(\om)$ have been established in the case where $\Omega$ is a square, $\triangle_g$ the Dirichlet-Laplacian operator on $\Omega$ and $\omega\subset\Omega$ is a finite union of squares.

\begin{theorem}[Computation of $\alpha^T(\om)$] \label{computation} 
We have
$$
\frac12 g_2^T(\mathring{\om}) \leq \alpha^T(\mathring{\om}) \leq \alpha^T(\om) \leq \alpha^T(\overline{\om}) \leq \frac12  g_2^T(\overline{\om}). 
$$
\end{theorem}

Let $\gamma$ be the support of a closed geodesic of $\Om$ and set $\om= \Om \setminus \gamma$ (open set). Then $\alpha^T(\om) = 1$ and $ g_2^T(\om) = 0$. 
Hence, the estimate given by Theorem \ref{computation} is not sharp. 

Note however that, if $\om$ is Jordan mesurable, i.e., if the Lebesgue measure of $\partial\omega=\overline\omega\setminus\mathring{\omega}$ is zero, then it follows from the definition of $C_T^{>N}$ that $C_T^{>N} (\om) = C_T^{>N}(\overline{\om})$ for every $N\in\N$.
As a consequence, Theorem \ref{computation} can be improved in that case by noting that 
$\frac12 g_2^T(\mathring{\overline{\om}}) \leq \alpha^T(\om)$, under additional regularity assumptions on $\omega$.

\begin{corollary} \label{corcomputation} 
If the measurable subset $\om$ satisfies the regularity assumption
\begin{equation*} \label{condition}
\Hun \qquad\qquad g_2^T(\Om \setminus (\overline{\om} \setminus \mathring{\om}))= 1
\end{equation*}
then
 $$
2 \alpha^T(\om)= g_2^T(\mathring{\om}) = g_2^T(\overline{\om})= g_2^T(\om).
 $$
\end{corollary}

Many measurable sets $\om$ satisfy Assumption $\Hun$. Geometrically speaking, $\Hun$ stipulates that $\om$ has no \textit{grazing ray}. We say that a ray $\gamma\in\Gamma$ is grazing $\omega$ if $\gamma(t)\in\partial\omega$ over a set of times of positive measure.

\noindent As a consequence of Corollary \ref{obser_carac}, Corollary \ref{corcomputation} and Theorem \ref{computation}, one has the following simple characterization of  observability.

\begin{corollary} \label{obser_carac_explicit}
Let $T>0$ and let $\om \subset \Om$ be a Lebesgue measurable subset of $\Omega$.
\begin{itemize}
\item[(i)] If $g_2^T(\mathring{\om})>0$ then $C_T(\om) >0$.
\item[(ii)] If $C_T(\overline\om)>0$ then $g_2^T(\overline\om)>0$.
\item[(iii)] Assume that $\om$ satisfies the regularity assumption $\Hun$. Then 
\begin{equation*}
g_2^{T}(\om) >0\ \Leftrightarrow\ C_T(\om) >0.
\end{equation*}
\end{itemize}
\end{corollary}

The first item above is already well known (see \cite{BardosLebeauRauch,rauch-taylor}): it says that, for $\om$ open, GCC implies observability. Indeed, the condition $g_2^T(\mathring{\om})>0$ is exactly GCC for $(\mathring{\om},T)$.
As already mentioned, the article \cite{BardosLebeauRauch} also deals with manifolds with boundary, which is not the case in this article. Recovering the boundary case by the method we present here would require a deeper study of the quantity $\alpha^T(\om)$ that we do not perform here.
We also mention \cite{BurqGerard}, where the authors prove that GCC is necessary and sufficient when replacing the characteristic function of $\omega$ by a continuous density function $a$ in all quantities introduced above. 

\medskip 

When there exist grazing rays the situation is more itricate. For instance, take $\Omega=\mathbb{S}^2$, the unit sphere of $\R^3$, and take $\omega$ the open Northern hemisphere. Then, the equator is a trapped ray (i.e., it never meets $\omega$) and is grazing $\omega$. Therefore we have $g_2^T(\omega)=0$ for every $T>0$, while $C_T(\omega)=g_1(\om)=g_1(\overline \om)=g_2^T(\overline\omega)=1/2$ for every $T\geq\pi$ (this follows immediately from computations done in \cite{Lebeau_JEDP1992}).

\medskip

Note also that $g_1(\om)>0$ is not sufficient to guarantee that \eqref{waveEqobs} is observable on $\omega$. For instance, take $\Om=\mathbb{T}^2$, the 2D torus, in which we choose $\omega$ as being the union of four triangles, each of them being at an corner of the square and whose side length is $1/2$. By construction, there are two trapped rays along $x=1/2$ and $y=1/2$ touching $\omega$ without crossing it over a positive duration. It follows that $g_2^T(\om)=g_2(\om)=C_T(\om)=0$ for every $T>0$. Moreover, simple computations show that $g_1(\om)>0$.

\medskip

From Theorem \ref{main_long_time} and Corollary \ref{corcomputation}, one gets the following asymptotic result.

\begin{corollary} \label{cormain}
If the measurable subset $\omega$ satisfies $\Hun$ then
\begin{equation*}
\lim_{T \to +\infty} \frac{C_T(\om)}{T}= \frac{1}{2}\min \left( g_1(\om), g_2(\overline{\om})\right).
\end{equation*}
\end{corollary}





\begin{remark}
The above result echoes a result by G. Lebeau that we recall hereafter.
In \cite{lebeau2}, the author considers the damped wave equation
\begin{equation}\label{eqdampedwave}
\partial_{tt}y(t,x)-\triangle_gy(t,x)+2a(x)\partial_ty(t,x)=0
\end{equation}
on a compact Riemannian manifold $\Omega$ with a $\mathcal{C}^\infty$ boundary, where the function $a(\cdot)$ is a smooth nonnegative function on $\Omega$. 
Given any $(y^0,y^1)\in H^1_0(\Omega)\times L^2(\Omega)$, for any $t\in\R$ we define
$$
E_{(y^0,y^1)}(t)=\int_\Omega (|\nabla y(t,x)|^2+(\partial_ty(t,x))^2)\, dv_g
$$
the energy at time $t$ of the unique solution $y$ of \eqref{eqdampedwave} such that $(y(0,\cdot),\partial_ty(0,\cdot))=(y^0,y^1)$.
Let $\om$ be any open set such that $a\geq \chi_\om$ almost everywhere in $\Omega$.
If $(\omega,T)$ satisfies GCC then there exist $\tau>0$ and $C>0$ such that
\begin{equation}\label{eq:2043}
E_{(y^0,y^1)}(t)\leq Ce^{-2\tau t} E_{(y^0,y^1)}(0)
\end{equation}
for all $(y^0,y^1)\in H^1_0(\Omega)\times L^2(\Omega)$ (see \cite{BardosLebeauRauch,Haraux,lebeau2}) and it is established in \cite[Theorem 2]{lebeau2} that the smallest decay rate $\tau(a)$ such that \eqref{eq:2043} is satisfied is 
$$
\tau(a)=\min \left( -\mu (\mathcal{A}_a),g_2(a)\right)
$$
where $g_2(a)$ is the geometric quantity defined by \eqref{def:g2} with $\chi_\om$ replaced by $a$, and $\mu(\mathcal{A}_a)$ is the spectral abscissa of the damped wave operator $\mathcal{A}_a=\begin{pmatrix}0 & \operatorname{Id}\\
\triangle_g & -2a(\cdot)\operatorname{Id}
\end{pmatrix}$.
\end{remark}

\begin{remark}[Probabilistic interpretation of the spectral quantity $g_1(\om)$]
The quantity $g_1(\om)$ can be interpreted as an averaged version of the observability constant $C_{T}(\om)$, where the infimum in \eqref{defCT} is now taken over random initial data. More precisely, let $(\beta_{1,j}^\nu)_{j\in\N^*}$ and $(\beta_{2,j}^\nu)_{j\in\N^*}$ be two sequences of Bernoulli random variables on a probability space $(\mathcal{X},\mathcal{A},\mathbb{P})$ such that
\begin{itemize}
\item for $m=1,2$, $\beta_{m,j}^\nu=\beta_{m,k}^\nu$ whenever $\lambda_j=\lambda_k$,
\item all random variables $\beta_{m,j}^\nu$ and $\beta_{m',k}^{\nu}$, with $(m,m')\in \{1,2\}^2$, $j$ and $k$ such that $\lambda_j\neq \lambda_k$, are independent,
\item there holds $\mathbb{P}(\beta_{1,j}^\nu=\pm 1)=\mathbb{P}(\beta_{2,j}^\nu=\pm 1)=\frac{1}{2}$ and $\mathbb{E}(\beta_{1,j}^\nu\beta_{2,k}^\nu)=0$,
for every $j$ and $k$ in $\N^*$ and every $\nu\in \mathcal{X}$.
\end{itemize}
Using the notation $\mathbb{E}$ for the expectation over the space $\mathcal{X}$ with respect to the probability measure $\mathbb{P}$, we claim that $\frac{T}{2}g_1(\om)$ is the largest nonnegative constant $C$ for which 
\begin{equation*}
C \Vert (y^0,y^1)\Vert^2_{L^2\times H^{-1}}
\leq \mathbb{E}\left(\int_0^T\int_\Omega \chi_\omega(x)\vert y^\nu(t,x) \vert ^2 \, dv_g \, dt\right)
\end{equation*}
for all $(y^0,y^1)\in L^2(\Omega)\times H^{-1}(\Omega)$, where $y^\nu$ is defined by
$$
y^\nu(t,x)=\sum_{j=1}^{+\infty}\left(\beta_{1,j}^\nu a_je^{i\lambda_{j}t}+\beta_{2,j}^\nu b_je^{-i\lambda_{j}t}\right)\phi_{j}(x),
$$
where the coefficients $a_j$ and $b_j$ are defined by
\begin{equation*}
\begin{split}
a_j &= \frac{1}{2}\left(\int_\Omega y^0(x) \phi_j (x)\, dv_g-\frac{i}{\lambda_j} \int_\Omega y^1(x) \phi_j (x)\, dv_g\right),\\
 b_j &= \frac{1}{2}\left(\int_\Omega y^0(x) \phi_j (x)\, dv_g+\frac{i}{\lambda_j} \int_\Omega y^1(x) \phi_j (x)\, dv_g\right)
\end{split}
\end{equation*}
\noindent for every $j\in\N^*$.
In other words, $y^\nu$ is the solution of the wave equation \eqref{waveEqobs} associated with the random initial data $y_{0}^{\nu}(\cdot)$ and $y_{1}^{\nu}(\cdot)$ determined by their Fourier coefficients $a_j^\nu=\beta_{1,j}^\nu a_j$ and $b_j^\nu=\beta_{2,j}^\nu b_j$.
This largest constant is called \emph{randomized observability constant} and has been defined in \cite[Section 2.3]{PTZobsND} and \cite[Section 2.1]{PTZ_optparab}. We also refer to \cite{PTZ_randObscont} for another deterministic interpretation of $\frac{T}{2}g_1(\om)$. 
\end{remark}

\begin{remark}[Extension of Corollary \ref{cormain} to manifolds with boundary.]
One could expect that a similar asymptotic to the one stated in Corollary \ref{cormain} holds for the Laplace-Beltrami operator on a manifold $\Omega$ such that $\partial\Omega\neq \emptyset$, with homogeneous Dirichlet boundary conditions. 
For instance, in the 1D case $\Omega=(0,\pi)$, it is prove in \cite[Lemma 1]{PTZObs1} by means of Fourier analysis that for every measurable set $\omega$
$$
\lim_{T\rightarrow +\infty}\frac{C_T(\om)}{T} = \inf_{j\in\N^*}\int_\omega\phi_j(x)^2\, dv_g=g_1(\om) \qquad \textrm{with}\quad \phi_j(x)= \frac{1}{\sqrt{\pi}}\sin (jx) .
$$
In higher dimension, the problem is more difficult because we are not able to compute explicitly $\alpha^T(\om)$ (see the proof of Theorem \ref{computation} where we use the Egorov theorem).
\end{remark}

\begin{center}
\fbox{\bf Spectral gap and consequences}\label{sec:gap}
\end{center}

\begin{theorem} \label{spectral_gap1}
Assume that the spectrum $(\lambda_j)_{j\in \N^*}$ satisfies the uniform gap property
\begin{quote} 
{\it (UG) \quad There exists $\gamma>0$ such that if $\lambda_j\neq \lambda_k$ then $|\lambda_j-\lambda_k|\geq \gamma$.}
\end{quote}
Then for every measurable subset $\omega$ of $\Omega$ we have
$$
\lim_{T\to +\infty} \frac{C_T(\om)}{T} 
= \frac{1}{2} g_1(\om).$$
\end{theorem}

As a consequence, thanks with Theorems \ref{main_long_time} and \ref{computation}, under $(UG)$ we have
\begin{equation} \label{g1leqg2}
 g_1(\om) \leq g_2(\overline{\om})
\end{equation}
for every measurable subset $\omega$ of $\Omega$. Note that, without spectral gap, such an inequality obviously does not hold true in general: take $\Omega$ the flat torus and $\omega$ a rectangle in the interior of $\Omega$ (see \cite{PTZobsND,PTZ_optparab} for various examples).

\begin{remark}
Note that the spectral gap assumption $(UG)$ is done for \emph{distinct} eigenvalues: it does not preclude multiplicity.
The assumption is satisfied for example for the sphere. Note that, under $(UG)$, the geodesic flow must be periodic (see \cite{Guillemin}), i.e., $\Omega$ is a Zoll manifold.
\end{remark}

 
 \begin{remark}[Application of Theorem \ref{spectral_gap1}]
Theorem \ref{spectral_gap1} applies in particular to the following cases:
\begin{itemize}
\item {\it The 1D torus $\mathbb{T}=\mathbb{R}/(2\pi)$.} The operator $\triangle_g=\partial_{xx}$ is defined on the subset of the functions of $H^2(\mathbb{T})$ having zero mean. All eigenvalues are of multiplicity $2$ and are given by $\lambda_j=j$ for every $j\in\N^*$ with eigenfunctions $e_j^1(x)=\sqrt{\frac{1}{\pi}}\sin(jx)$ and $e_j^2(x)=\sqrt{\frac{1}{\pi}}\cos(jx)$. The spectral gap is $\gamma=1$ and we compute
\begin{equation*}
\begin{split}
\lim_{T\to +\infty} \frac{C_T(\om)}{T} &=\frac{1}{\pi}\inf_{j\in\N^*}\inf_{\alpha \in [0,1]}\int_\omega \left(\sqrt{\alpha}\sin (jx)+\sqrt{1-\alpha}\cos(jx)\right)^2\, dx \\
&=\frac{1}{\pi}\left(\frac{|\omega|}{2}-\sup_{j\in\N^*}\sqrt{\left(\int_\omega \sin (2jx)\, dx\right)^2+\left(\int_\omega \cos (2jx)\, dx\right)^2}\right)
\end{split}
\end{equation*}
\item {\it The unit sphere $\mathbb{S}^{n}$ of $\R^{n+1}$.} The operator $\triangle_g$ is defined from the usual Laplacian operator on the Euclidean space $\R^{n+1}$ by the formula  $\triangle_g=r^2\triangle_{\R^{n+1}}-\partial_{rr}-\frac{n}{r}\partial_r$ where $r=\Vert x\Vert_{\R^{n+1}}$ for every $x\in\R^{n+1}$. Its eigenvalues are $\lambda_k=k(k+n-1)$ where $k\in \N$. The multiplicity of $\lambda_k$ is $k(k+n-1)$ and the space of eigenfunctions is the space of homogeneous harmonic polynomials\footnote{An orthogonal basis of spherical harmonics is given by
$$
Y_{l_1, \dots l_{n}} (\theta_1, \dots \theta_{n}) = \frac{1}{\sqrt{2\pi}} e^{i l_1 \theta_1} \prod_{j = 2}^{n} \widetilde{P}^{l_{n-1}}_{l_j,j} (\theta_j)
$$
where the indices are integers satisfying $|l_1| \leq l_2 \leq ... \leq l_{n}$ and the eigenvalue is $-l_{n}(l_{n} + n-1)$. The functions in the product are defined by
$$
 \widetilde{P}^l_{L,j} (\theta) = \sqrt{\frac{2L+j-1}{2} \frac{(L+l+j-2)!}{(L-l)!}} \sin^{\frac{2-j}{2}} (\theta) P^{-(l + \frac{j-2}{2})}_{L+\frac{j-2}{2}} (\cos \theta),
$$
where, for two real numbers $\nu$ and $\mu$, the function $P^{-\mu}_{\nu}$ is the associated Legendre function of the first kind defined by
$$
P_\nu^{-\mu}(x)=\frac{1}{\Gamma(1+\mu)}\left(\frac{1-x}{1+x}\right)^{\mu/2}F\left(-\nu,\nu+1,1+\mu,\frac{1-x}{2}\right),
$$ 
where $\Gamma$ is the Euler's Gamma function and $F$ is the hypergeometric function (see e.g. \cite{Higuchi}).} of degree $k$.
As a result, we compute
$$
\lim_{T\to +\infty} \frac{C_T(\om)}{T} =\inf_{k\in\N}\inf_{\phi\in \mathcal{H}_k}\frac{\int_\omega |\phi(x)|^2\, dx}{\int_{\mathbb{S}^n} |\phi(x)|^2\, dx},
$$
where $\mathcal{H}_k$ is the space of homogeneous harmonic polynomials of degree $k$.
\end{itemize}
\end{remark}

As a byproduct of Theorem \ref{spectral_gap1}, we recover a well known result on the existence of quantum limits supported by closed rays. Recall that a quantum limit for $-\triangle_{g}$ is a probability measure given as a weak limit (in the space of Radon measures) of the sequence of measures $(\phi_j(x)^2\, dx)_{j\in\N^*}$.

\begin{corollary} \label{qumeasure}
Under $(UG)$, for any (closed) ray $\gamma\in\Gamma$ there exists a quantum limit supported on $\gamma$.
\end{corollary}

This is exactly one of the main results of \cite{Macia} which extends a result in \cite{JZ} on the sphere. As a consequence also noted in \cite{Macia}, under the additional assumption that $\Om$ is a Zoll manifold with maximally degenerate Laplacian, any measure invariant under the geodesic flow is a quantum limit. The converse is not true (see \cite{MaciaRiviere_CMP2016}).

\section{Proofs} \label{sec:split}

This section is devoted to prove the results stated in the latter section. In the next paragraph, we establish many results which imply all the results stated in the Introduction. More precisely, 

\begin{itemize}
 \item Theorem \ref{main_finite_time} is a consequence of Lemma \ref{obvious} and Theorem \ref{compacity};
 \item Corollary \ref{obser_carac} is proved in Section \ref{observability_cor};
 \item Theorem \ref{main_long_time} is proved in Section \ref{longtime};
 \item Corollary \ref{corcomputation} is a consequence of Proposition \ref{conditionH};
 \item Corollaries \ref{obser_carac_explicit} and \ref{cormain} follow from the above the results;
 \item Theorem \ref{spectral_gap1} is proved in Section \ref{proofug}.
\end{itemize}

\subsection{Preliminaries and notations} \label{Prelim}
Let us set $\Lambda=\sqrt{-\triangle}$. Given any $(y^0,y^1)\in L^2(\Omega)\times H^{-1}(\Omega)$, standing for initial conditions for the wave equation, we set 
\begin{equation}\label{def:y+y-}
y^+ = \frac{1}{2}(y^0-i\Lambda^{-1}y^1)\in L^2(\Omega)\quad\text{ and }\quad y^- = \frac{1}{2}(y^0+i\Lambda^{-1}y^1)\in L^2(\Omega).
\end{equation}
The mapping $(y^0,y^1)\in L^2(\Omega)\times H^{-1}(\Omega) \mapsto (y^+,y^-)\in L^2(\Omega)\times L^2(\Omega)$ is an isomorphism, and $\Vert (y^0,y^1)\Vert^2_{L^2\times H^{-1}} = 2 ( \Vert y^+\Vert^2_{L^2}+\Vert y^-\Vert^2_{L^2} )$. The unique solution $y$ of the wave equation \eqref{waveEqobs} associated to the pair of initial data $(y^0,y^1)$ belongs to $C^0(0,T;L^2(\Omega))\cap C^1(0,T;H^{-1}(\Omega))$ and writes $y(t) = e^{it\Lambda}y^+ + e^{-it\Lambda}y^-$.

By definition, we have 
$$
C_T(\om) = \frac{1}{2} \inf_{\Vert y^+\Vert^2_{L^2}+\Vert y^-\Vert^2_{L^2}=1} \int_0^T \int_\Omega \chi_\om(x) \left\vert (e^{it\Lambda} y^+)(x) + (e^{-it\Lambda} y^-)(x) \right\vert^2 \, dv_g(x)\, dt.
$$
Let $a:M \to \R$ be any measurable nonnegative function. We denote (with a slight abuse of notation) by $C_T(a)$ the quantity
$$
C_T(a) = \frac{1}{2} \inf_{\Vert y^+\Vert^2_{L^2}+\Vert y^-\Vert^2_{L^2}=1} \int_0^T \int_\Omega a(x) \left\vert (e^{it\Lambda} y^+)(x) + (e^{-it\Lambda} y^-)(x) \right\vert^2 \, dv_g(x)\, dt.
$$
This way, one has $C_T(\omega)=C_T(\chi_\omega)$.

We have 
\begin{equation}\label{eq1633meuse}
\begin{split}
& \frac{1}{T}\int_0^T \int_\Omega a \vert e^{it\Lambda} y^+ + e^{-it\Lambda} y^- \vert^2 dv_g\, dt \\
&= \frac{1}{T}\int_0^T \big( \langle a e^{it\Lambda}y^+, e^{it\Lambda}y^+\rangle + \langle a e^{-it\Lambda}y^-, e^{it\Lambda}y^-\rangle  
+ \langle a e^{it\Lambda}y^+, e^{-it\Lambda}y^-\rangle + \langle a e^{-it\Lambda}y^-, e^{it\Lambda}y^+
\rangle  \big)dv_g\, dt   \\
&=  \left\langle \frac{1}{T}\int_0^Te^{-it\Lambda} a e^{it\Lambda}\, y^+, y^+\right\rangle + \left\langle \frac{1}{T}\int_0^T e^{it\Lambda} a e^{-it\Lambda}\, dt\ y^-, y^-\right\rangle \\
&\qquad\qquad  + \left\langle \frac{1}{T}\int_0^Te^{it\Lambda} a e^{it\Lambda}\, dt\ y^+, y^-\right\rangle + \left\langle \frac{1}{T}\int_0^Te^{-it\Lambda} a e^{-it\Lambda}\, dt\ y^-, y^+\right\rangle
\end{split}
\end{equation}
where $\langle \cdot, \cdot \rangle$ is the scalar product in $L^2(\Om,v_g)$. Here, $a$ is considered as an operator by multiplication.
This formula suggests to introduce the operators $\bar A_T$ and $\bar B_T$ defined by
$$
\bar A_T(a)=\frac{1}{T}\int_0^T e^{-it\Lambda} a e^{it\Lambda}\, dt \; \hbox{ and } \; 
\bar B_T(a)=\frac{1}{T}\int_0^T e^{it\Lambda} a e^{it\Lambda} dt,
$$ 
so that 
\begin{eqnarray} \label{CT}
 C_T(a) =  \inf_{\Vert y^+\Vert^2_{L^2}+\Vert y^-\Vert^2_{L^2}=1}  J^a_{T}(y^+,y^-) 
\end{eqnarray}
with 
\begin{equation*}
  J^a_{T}(y^+,y^-)   =  \frac{1}{2}  \Big( \left\langle \bar A_T(a) y^+ , y^+ \right\rangle +  \left\langle \bar A_{-T}(a) y^- , y^- \right\rangle   
 +  \left\langle \bar B_T(a) y^+ , y^- \right\rangle  +  \left\langle \bar B_{-T}(a) y^- , y^+ \right\rangle \Big).
\end{equation*}
Given any $N \in \N$, we extend similarly the definition of $C_T^{>N}(\omega)$ by defining 
\begin{multline*}
C_T^{> N}(a)  = \inf  \{J_T^{a}(y^0,y^1)\mid \langle y^i,\phi_j\rangle_{(H^i)',H^i}=0, \quad \forall i=0,1, \quad \forall j= 1,\ldots,N\\
(y^0,y^1) \in L^2(\Omega)\times H^{-1}(\Omega) \setminus\{(0,0)\} \}  
\end{multline*} 
and $\alpha^T(a)= \lim_{N  \to + \infty}  \frac{1}{T} C_T^{> N}(a)$. In what follows, the index $N$ means that we consider initial conditions involving eigenmodes of index larger than $N$. More precisely, if $y \in H^{-1}(\Om)$, $\langle y_N, \phi_j \rangle_{H^{-1},H^1} = 0$ for every $j \leq N$. The same reasoning as above to obtain \eqref{CT} yields 
\begin{eqnarray} \label{CTN}
J_T^{a}(y^0,y^1) = \frac{1}{2}  J^a_{T}(y^+_N,y^-_N) .
\end{eqnarray}

\subsection{Comments on Assumption $\Hun$}
\begin{proposition} \label{conditionH}
Under $\Hun$ we have $g_2(\mathring{\om})= g_2(\overline{\om})$. 
\end{proposition}

\begin{proof} 
Let  $\ep>0$. Without loss of generality we assume that $\om$ is open. By definition of the infimum in the definition of $g_2^T(\om)$, for every $\varepsilon>0$ there exists a ray $\gamma\in\Gamma$ such that 
$$\begin{aligned} 
g_2^T(\mathring{\om}) + \ep & \geq \frac{1}{T} \int_{0}^{T} \chi_{\om} (\gamma(t) ) \, dt  = \frac{1}{T} \int_{0}^{T} \chi_{\overline{\om}} (\gamma(t) )\,  dt - \frac{1}{T} \int_{0}^{T} \chi_{\overline{\om} \setminus \om}(\gamma(t) ) \, dt \\
   & = \frac{1}{T} \int_{0}^{T} \chi_{\overline{\om}} (\gamma(t) ) \, dt + \frac{1}{T} \int_{0}^{T} \chi_{\Om \setminus (\overline{\om} \setminus \om)} (\gamma(t) ) \, dt -1 \\
      & \geq g_2^T(\overline{\om}) + g_2^T( \Om \setminus (\overline{\om} \setminus \om)) -1 
    \geq g_2^T(\overline{\om})
  \end{aligned}
$$
and thus $g_2^T(\mathring{\om}) \geq g_2^T(\overline{\om})$. The converse inequality is obvious.
\end{proof}

\subsection{Upper bound for $C_T$} 

\begin{lemma} \label{obvious} 
For every Lebesgue measurable subset $\omega$ of $M$, one has
$$
 \frac{C_T(\omega)}{T} \leq \min\left( \frac{1}{2} g_1(\omega),\alpha^T(\omega)\right).
$$
\end{lemma}

\begin{proof}
By considering particular solutions of the form $e^{it\Lambda}\phi_j$ for a given $j\in\N^*$, we obtain $\frac{C_T(\omega)}{T} \leq \frac{1}{2} g_1(\omega)$. Besides, we have $C_T(\omega) \leq C_T^{>N}(\omega)$ and letting $N$ tend to $+\infty$, we get
 $C_T(\omega) \leq \al^T(\omega)$.
\end{proof}

\subsection{The high-frequency observability constant $\al^T$}
The quantity $g_2^T$ has been defined for measurable subsets $\om$, but similarly to what has been done in Section \ref{Prelim}, we extend its definition to arbitrary measurable nonnegative bounded functions $a:M \to \R$, by setting
$$
g_2^T(a) =  \inf_{\gamma\in\Gamma} \frac{1}{T} \int_0^T a(\gamma(t))\, dt .
$$
With this notation, we have $g_2^T(\chi_\om)=g_2^T(\om)$, with a slight abuse of notation.

\begin{theorem} \label{main_a}
For every continuous nonnegative function $a:M \to \R$, we have
 $$\alpha^T(a) = \frac{1}{2} g_2^T(a).$$ 
\end{theorem}

\begin{proof}
We first assume that the function $a:M \to \R$ is smooth and thus can be considered as the symbol of an pseudo-differential  $\Op(a)$ of order $0$ corresponding to the multiplication by $a$. We have
 $$\bar A_T(a)=\frac{1}{T}\int_0^T e^{-it\Lambda} \Op(a) e^{it\Lambda}\, dt \; \hbox{ and } \; 
\bar B_T(a)=\frac{1}{T}\int_0^T e^{it\Lambda} \Op(a) e^{it\Lambda} dt.$$ 
According to the Egorov theorem (see \cite{Egorov,zworski}), the pseudo-differential operators $\bar A_T$ and $\bar A_{-T}$ are of order $0$ and their principal symbols are respectively 
  $$
  \bar a_T = \frac{1}{T}\int_0^T a\circ\varphi_t\, dt \; \hbox{ and } \bar a_{-T} = \frac{1}{T}\int_0^T a\circ\varphi_{-t}\, dt,
  $$ 
 where $(\varphi_t)_{t\in \R}$ is the Riemannian geodesic flow. Besides, 
 $$
 \bar B_T(a)=\frac{1}{T}\int_0^Te^{it\Lambda} \Op(a) e^{it\Lambda}\, dt \quad \hbox{ and } \quad \bar B_{-T}(a)=\frac{1}{T}\int_0^Te^{-it\Lambda} \Op(a) e^{-it\Lambda}\, dt$$
 are pseudo-differential operators of order $-1$ and hence are compact (see \cite[Section \ref{Prelim}]{Dehman_LeRousseau_leautaud}). 

Defining $y_+$ by \eqref{def:y+y-} and $y^+_N$ as in \eqref{CTN}, we compute (as in \eqref{eq1633meuse})
\begin{multline*}
\frac{1}{T}\int_0^T \int_\Omega a \vert e^{it\Lambda} y_N^+ + e^{-it\Lambda} y_N^- \vert^2 dv_g\, dt \\
= \left\langle \bar A_T(a) y^+_N , y^+_N \right\rangle +  \left\langle \bar A_{-T}(a) y^-_N , y^-_N \right\rangle  +\left\langle \bar B_T(a) y^-_N , y^+_N \right\rangle +  \left\langle \bar B_{-T}(a) y^+_N , y^-_N \right\rangle .
\end{multline*}
Considering for instance the first term at the right-hand side, we have
$$
\left\langle \bar A_T(a) y^+_N , y^+_N \right\rangle  = \left\langle \frac{1}{T}\int_0^Te^{-it\Lambda} \Op(a) e^{it\Lambda}\, dt\ y^+_N, y^+_N\right\rangle = \langle \Op(\bar a_T) y^+_N, y^+_N\rangle  + \langle K_T y^+_N, y^+_N\rangle
$$
where $K_T$ is a pseudo-differential operator of order $-1$ (depending on $a$) and thus $\vert\langle K_T y^+_N, y^+_N\rangle\vert\leq \Vert K_T\Vert \Vert y^+_N\Vert_{L^2}\Vert y^+_N\Vert_{H^{-1}}$. 
It follows from \eqref{CTN} that
\begin{equation*}
\frac{1}{T}C_T^{>N}(a) = \frac{1}{2} \inf_{\Vert y^+_N\Vert^2_{L^2}+\Vert y^-_N\Vert^2_{L^2}=1 }  \Big(  \langle \Op(\bar a_T) y^+_N, y^+_N\rangle + \langle \Op(\bar a_{-T}) y^-_N, y^-_N\rangle \Big) + \mathrm{o}(1)\qquad \text{as }N\rightarrow +\infty.
\end{equation*}

Let us  first prove that $\alpha^T(a)\geq \frac{1}{2}g_2^T(a)$. Denote by $S^* \Om$ the unit cotangent bundle over $ \Om$. 
By definition, we have $\bar a_T(x,\xi)\geq g_2^T(a)$ for every $(x,\xi)\in S^*\Om$ (and similarly, $\bar a_{-T}(x,\xi)\geq g_2^T(a)$), and since the symbol $\bar a_T$ is real and of order $0$, it follows from the G\aa rding inequality (see \cite{zworski}) that for every $\varepsilon>0$ there exists $C_\varepsilon>0$ such that
$$
\langle \Op(\bar a_T) y^+_N, y^+_N\rangle \geq  (g_2^T(a)-\varepsilon) \Vert y^+_N\Vert^2_{L^2} - C_\varepsilon \Vert y^+_N\Vert^2_{H^{-1/2}} 
$$
for every $y^+_N\in L^2(\Omega)$ (actually, one can even take $\varepsilon=0$ by using a positive quantization, for instance $\Op^+$).
Since the spectral expansion of $y_N^+$ involves only modes with indices larger than $N$, we have $\Vert y^+_N\Vert^2_{H^{-1/2}} \leq \frac{1}{\lambda_N}\Vert y^+_N\Vert^2_{L^2}$ and it follows that, when considering the infimum over all possible $y_N^\pm$ of $L^2$ norm equal to $1$, all remainder terms provide a remainder term $\mathrm{o}(1)$ as $N\rightarrow +\infty$, uniformly with respect to $y_N^\pm$.
We conclude that $C_T^{>N}(a)\geq \frac{1}{2} g_2^T(a) + \mathrm{o}(1)$, and thus $\alpha^T(a)\geq \frac{1}{2} g_2^T(a)$.

Let us now prove that $\alpha^T(a)\leq \frac{1}{2} g_2^T(a)$. The idea is to choose some appropriate  $y_N^+\in L^2(\Omega)$, and $y_N^-=0$, and to write that $\frac{1}{T}C_T^{>N}(a) \leq \frac{1}{2} \langle \Op(\bar a_T) y^+_N, y^+_N\rangle  + \mathrm{o}(1)$.
The choice of an appropriate $y_N^+$ is guided by the following lemma on coherent states.

\begin{lemma}\label{lem_coherent_state}
Let $x_0\in\R^n$, $\xi_0\in\R^n$, and $k\in\N^*$. We define the coherent state
$$
u_k(x) = \left(\frac{k}{\pi}\right)^\frac{n}{4} e^{ik(x-x_0).\xi_0-\frac{k}{2}\Vert x-x_0\Vert^2}.
$$
Then $\Vert u_k\Vert_{L^2}=1$, and for every symbol $a$ on $\R^n$ of order $0$, we have
$$
\mu_k(a) = \langle\Op(a)u_k,u_k\rangle_{L^2} = a(x_0,\xi_0) + \mathrm{o}(1),
$$
as $k\rightarrow +\infty$. In other words, $(\mu_k)_{k\in \N}$ converges in the sense of measures to $\delta_{(x_0,\xi_0)}$.
\end{lemma}

\noindent Admitting temporarily this (well known) lemma, we are going to define $y_N^+$ as an approximation of $u_k$, having only frequencies larger than $N$. 
Let $(x_0,\xi_0) \in S^*M$ be a minimizer of $\bar a_T$, i.e., $g_2^T(a) = \min \bar a_T = \bar a_T(x_0,\xi_0)$.
We consider the above solution $u_k$, defined on $M$ in a local chart around $(x_0,\xi_0)$ (we multiply the above expression by a function of compact support taking the value $1$ near $(x_0,\xi_0)$, and we adapt slightly the constant so that we still have $\Vert u_k\Vert_{L^2}=1$). Note that $\int_\Omega u_k \, dv_g = \frac{2^\frac{n}{2}\pi^\frac{n}{4}}{k^\frac{n}{4}}$.
Now, we set
$$
\pi_N u_k = \sum_{j=1}^N \langle u_k,\phi_j\rangle \phi_j = \sum_{j=1}^N \int_\Omega u_k(x)\phi_j(x)\, dx\ \phi_j dv_g(x).
$$
By usual Sobolev estimates and by the Weyl law, there exists $C>0$ such that $\Vert\phi_j\Vert_{L^\infty(\Omega)}\leq C \lambda_j^{\frac{n}{2}}$ and $\lambda_j\sim j^{\frac{2}{n}}$ for every $j\in\N^*$, hence $\Vert\phi_j\Vert_{L^\infty(\Omega)}\leq C j$. We infer that 
$$
\vert\langle u_k,\phi_j\rangle\vert\leq CN \int_\Omega \vert u_k\vert\leq C2^\frac{n}{2}\pi^\frac{n}{4}\frac{N}{k^\frac{n}{4}} dv_g(x)
$$ 
for every $j\leq N$.

Let $\varepsilon>0$ be arbitrary. Choosing $k$ large enough so that $C 2^\frac{n}{2}\pi^\frac{n}{4}\frac{N^2}{k^\frac{n}{4}}\leq \varepsilon$, we have $\Vert \pi_N u_k\Vert_{L^2}\leq\varepsilon$.
We set $y_N^+ = u_k - \pi_N u_k$.
We have 
\begin{equation*}
\langle \Op(\bar a_T) y^+_N, y^+_N\rangle
= \underbrace{\langle \Op(\bar a_T) u_k, u_k\rangle}_{\simeq g_2^T(a)} + \underbrace{\langle \Op(\bar a_T) \pi_N u_k, \pi_N u_k\rangle}_{\leq \varepsilon^2\max\bar a_T} - \underbrace{\langle \Op(\bar a_T) \pi_N u_k, u_k\rangle}_{\vert\cdot\vert\leq \varepsilon\max\bar a_T} - \underbrace{\langle \Op(\bar a_T) u_k, \pi_N u_k\rangle}_{\vert\cdot\vert\leq \varepsilon\max\bar a_T} 
\end{equation*}
and the conclusion follows.
\end{proof}

\begin{proof}[Proof of Lemma \ref{lem_coherent_state}.]
This lemma can be found for instance in \cite[Chapter 5, Example 1]{zworski}. We include a proof for the sake of completeness. 
First of all, we compute\footnote{Here, we use the fact that $\int_{\R^n} e^{-\alpha\Vert x\Vert^2}\, dx = \left(\frac{\pi}{\alpha}\right)^\frac{n}{2}$.}
$
\Vert u_k\Vert_{L^2}^2 = \left(\frac{k}{\pi}\right)^\frac{n}{2} \int e^{-\frac{k}{2}\Vert x-x_0\Vert^2}\, dx = 1.
$
Now, by definition, we have
\begin{multline*}
\langle\Op(a)u_k,u_k\rangle_{L^2}
= \int \Op(a)u_k(x) \overline{u_k(x)}\, dx 
= \frac{1}{(2\pi)^n} \iiint e^{i(x-y).\xi} a(x,\xi) u_k(y) \overline{u_k(x)}\, dx\, dy\, d\xi \\
= \frac{k^n}{(2\pi)^n} \iiint e^{ik(x-y).\xi} a(x,\xi) u_k(y) \overline{u_k(x)}\, dx\, dy\, d\xi
\end{multline*}
by the change of variable $\xi\mapsto k\xi$, and using the homogeneity of $a$. Then we get
\begin{equation*}
\begin{split}
\langle\Op(a)u_k,u_k\rangle_{L^2}
&= \frac{k^\frac{3n}{2}}{2^n\pi^\frac{3n}{2}} \iiint a(x,\xi) e^{ik(x-y).\xi} e^{ik(y-x).\xi_0} e^{-\frac{k}{2} (  \Vert x-x_0\Vert^2 + \Vert y-x_0\Vert^2 )} \, dx\, dy\, d\xi \\
&= \frac{k^\frac{3n}{2}}{2^n\pi^\frac{3n}{2}} \iint a(x,\xi) e^{-\frac{k}{2} \Vert x-x_0\Vert^2} \int e^{ik(x-y).\xi} e^{ik(y-x).\xi_0} e^{-\frac{k}{2} \Vert y-x_0\Vert^2 } \, dy \, dx\, d\xi .
\end{split}
\end{equation*}
Noting that $\mathcal{F}(e^{-\alpha\Vert x\Vert^2})(\xi) = \left(\frac{\pi}{\alpha}\right)^\frac{n}{2}e^{-\frac{\Vert\xi\Vert^2}{4\alpha}}$, we obtain
\begin{multline*}
\int e^{ik(x-y).\xi} e^{ik(y-x).\xi_0} e^{-\frac{k}{2} \Vert y-x_0\Vert^2 } \, dy
= e^{ik(x-x_0).(\xi-\xi_0)} \int e^{-ik(y-x_0).(\xi-\xi_0)} e^{-\frac{k}{2} \Vert y-x_0\Vert^2 } \, dy \\
= e^{ik(x-x_0).(\xi-\xi_0)} \int e^{-iky.(\xi-\xi_0)} e^{-\frac{k}{2} \Vert y\Vert^2 } \, dy 
= e^{ik(x-x_0).(\xi-\xi_0)} \mathcal{F}(e^{-\frac{k}{2}\Vert y\Vert^2})(k(\xi-\xi_0)) \\
= \left( \frac{2\pi}{k}\right)^\frac{n}{2} e^{ik(x-x_0).(\xi-\xi_0)} e^{-\frac{k}{2}\Vert \xi-\xi_0\Vert^2}
\end{multline*}
and therefore,
\begin{equation*}
\begin{split}
\langle\Op(a)u_k,u_k\rangle_{L^2}
&= \frac{k^n}{2^\frac{n}{2}\pi^n} \iint a(x,\xi) e^{ik(x-x_0).(\xi-\xi_0)} e^{-\frac{k}{2} (\Vert x-x_0\Vert^2 + \Vert \xi-\xi_0\Vert^2 )}\, dx\, d\xi  \\
&= \frac{k^n}{2^\frac{n}{2}\pi^n} a(x_0,\xi_0) \iint e^{ik(x-x_0).(\xi-\xi_0)} e^{-\frac{k}{2} (\Vert x-x_0\Vert^2 + \Vert \xi-\xi_0\Vert^2 )}\, dx\, d\xi  +  \mathrm{o}(1) \\
&= c_n a(x_0,\xi_0) +  \mathrm{o}(1)
\end{split}
\end{equation*}
as $k\rightarrow +\infty$. Moreover, taking $a=1$ above, we see that
$c_n = \iint e^{ikx.\xi} e^{-\frac{k}{2} (\Vert x\Vert^2 + \Vert \xi\Vert^2 )}\, dx\, d\xi = 1$.
The lemma is proved.
\end{proof}

It remains to extend the statement to the case where $a$ is continuous only. It is obvious from the definitions of $\alpha^T$ and $g_2^T$ that if $(a_k)_{k\in\N}$ is sequence of nonnegative smooth functions converging uniformly to $a$, then 
$$
\lim_{k\to +\infty} \al^T(a_k)=\al^T(a) \; \hbox{ and } \; \lim_{k\to +\infty} g_2^T(a_k)=g_2(a).
$$
Indeed, this is a consequence of the two following facts: 
\begin{itemize}
\item the supremum of $\frac{1}{T} \int_0^T \int_\Omega |a_k-a| y^2 \, dv_g\, dt$ over the set of all functions $y$ satisfying $\|y\|_{L^2}=1$ tends to $0$ as $k\to +\infty$;
\item the supremum of $\frac{1}{T} \int_0^T |a_k-a|(\gamma(t)) dt$ over the set of all rays $\gamma$ tends to $0$ as $k\to +\infty$.
\end{itemize}
The theorem is proved.

\begin{remark}
Note that $e^{it\Lambda}u_k$ (or, accordingly, $e^{it\Lambda}(u_k-\pi_Nu_k)$) is a half-wave Gaussian beam along the geodesic $\varphi_t(x_0,\xi_0)$.
Indeed, for any symbol of order $0$, recalling that $A_t = e^{-it\Lambda}\Op(a)e^{it\Lambda}$ has $a_t = a\circ\varphi_t$ as principal symbol, we have $\langle\Op(a)e^{it\Lambda}u_k,e^{it\Lambda}u_k\rangle = \langle A_t u_k,u_k\rangle = \langle \Op(a_t) u_k,u_k\rangle+\mathrm{o}(1) = a_t(x_0,\xi_0)+\mathrm{o}(1)$ (by Lemma \ref{lem_coherent_state}), which means that $e^{it\Lambda}u_k$ is microlocally concentrated around $\varphi_t(x_0,\xi_0)$.
\end{remark}

\subsection{Proof of Theorem \ref{computation}}
Consider an increasing sequence $(h_k)_{k\in\N}$ of continuous functions  such that $0 \leq h_k \leq 1$ in $\Om$, $h_k(x)=0$ if $\hbox{dist}(x, \Om \setminus \mathring{\om}) \leq \frac{1}{k}$ and $h_k(x) = 1$ if   $\hbox{dist}(x, \Om \setminus \mathring{\om}) \geq \frac{2}{k}$. Note that $0\leq h_k\leq h_{k+1}\leq \chi_{\mathring\om}$ for every $k\in\N$. Let us prove that 
\begin{equation} \label{ineg_lemmalimsup}
g_2^T(\mathring\om) = \lim_{k\to +\infty} g_2^T(h_k ).
\end{equation}
The fact that  $g_2^T(\mathring{\om}) \geq  \limsup_{k\to +\infty} g_2^T(h_k )$ is obvious since $\chi_{\mathring{\om}} \geq h_k$ for all $k\in \N$. 
Consider a sequence of rays $\gamma_k:[0,T] \to \Om$ such that 
\begin{eqnarray} \label{g2hk}
 g_2^T(h_k) \geq \frac{1}{T} \int_0^T h_k(\gamma_k(t))\, dt + \operatorname{o}(1)\quad \textrm{as }k\to +\infty.
\end{eqnarray}
The set of rays is compact since each ray is determined by it position $x \in \Om$ at time $0$ and its derivative at time $0$ which lies on the unit cotangent bundle of $\Om$. Hence there exists $\gamma:[0,T] \to \Om$ such that 
$\gamma_k \to \gamma$ uniformly on $[0,T]$. 
For any $t \in [0,T]$, one has 
$$
\liminf_{k\to +\infty} h_k(\gamma_k(t)) \geq \chi_{\mathring{\om}}(\gamma(t)). 
$$
Indeed, if $\gamma(t) \in {\mathring{\om}}$, then since ${\mathring{\om}}$ is open, $h_k(\gamma_k(t))=1 = \chi_{\mathring{\om}}(\gamma(t))$ as soon as $k$ is large enough. If $\gamma(t) \not\in {\mathring{\om}}$, the inequality is obvious since $\chi_{\mathring{\om}}(\gamma(t)) = 0$. 
By dominated convergence, we infer from \eqref{g2hk} that
$$
g_2^T(h_k) \geq \frac{1}{T} \int_0^T h_k(\gamma_k(t))\, dt + \operatorname{o}(1) \geq \frac{1}{T} \int_0 \chi_{\mathring{\om}}(\gamma(t))\, dt + \operatorname{o}(1) \geq g_2^T({\mathring{\om}}) +\operatorname{o}(1) \quad \textrm{as }k\to +\infty ,
$$
which proves \eqref{ineg_lemmalimsup}. 

Using that the sequence $(h_k)_{k\in \N}$ is increasing and since each $h_k$ is continuous, we obtain
$$
\frac{1}{2} g_2^T(\mathring{\om}) = \lim_{k\to +\infty} \frac{1}{2}  g_2^T(h_k) = \lim_{k\to +\infty} \al^T(h_k) \leq \al^T(\mathring{\om}) \leq  \al^T(\om) \leq \al^T(\bar\om).
$$

To conclude the proof of Theorem \ref{computation}, it remains to prove that 
\begin{equation} \label{alleqom}
 \al^T(\bar\om) \leq \frac{1}{2} g_2^T(\bar \om).
\end{equation}
The proof of this inequality uses exactly the same reasoning as the one used to prove $\frac{1}{2} g_2^T(\mathring{\om})  \leq \al^T(\mathring{\om})$. Indeed, we consider a decreasing sequence of continuous functions $(h_k)_{k\in \N}$ converging pointwisely to $\chi_{\bar\omega}$, and therefore, we have $\alpha^T(\bar\omega)\leq \alpha^T(h_k) = \frac{1}{2} g_2^T(h_k)$ and $\lim_{k\to \infty} g_2^T(h_k)= g_2^T(\bar\omega)$. We conclude as previously that \eqref{alleqom} is true.

\subsection{Low frequencies compactness property}
According to Lemma \ref{obvious}, one has $\frac{1}{T}C_T(\omega)\leq \min \left(  \frac{1}{2} g_1(\omega),  \alpha^T(\omega) \right)$.

\begin{proposition} \label{compacity}
 If $\frac{1}{T} C_T(\omega) < \alpha^T(\omega)$ then $C_T(\omega)$ is reached, i.e., the infimum defining $C_T(\omega)$ is in fact a minimum.
\end{proposition}

\begin{proof}  
Let  $(Y_k)_{k\in \N} = (y^+_k,y^-_k)_{k\in \N} \in (L^2(\Om) \times L^2(\Om))^\N$ be  such that  
$$
\lim_{k\to +\infty} J^{\chi_\omega}_{T} (Y_k) = \frac{C_T(\om)}{T}
$$
where $J^{\chi_\omega}_{T}(y)$ is defined in Section \ref{Prelim} (see \eqref{CT}) with $\| y^+_k \|_{L^2}^2 + \| y^+_k \|_{L^2}^2 =1$ for every $k\in \N$.

Since  the sequences $(y^\pm_k)_{k\in \N}$ are bounded in $L^2$, they converges weakly to an element $ y^\pm_\infty \in L^2$ up to a subsequence. Therefore, we write $Y_k = Y_\infty + Z_k$ with $Y_\infty= (y^+_\infty,y^-_\infty)$ and $Z_k=(z^+_k,z^-_k)$ such that $Z_k \rightharpoonup 0$ in $L^2(\Om) \times L^2(\Om)$. Note that we use the norm in $L^2 \times L^2$ defined by $\| (y,z) \|^2  = \|y\|_{L^2}^2 + \|z\|_{L^2}^2$.
With this notations, the weak convergence of $Z_k$ to $0$ yields 
\begin{equation} \label{normyk}
 1 = \Vert Y_k\Vert^2 = \Vert Y_\infty\Vert^2 + \Vert Z_k\Vert^2 + \mathrm{o}(1)
\end{equation}
and 
\begin{equation} \label{normJyk}
 J^{\chi_\omega}_{T}(Y_k) = J^{\chi_\omega}_{T}(Y_\infty)  + J^{\chi_\omega}_{T}(Z_k)+ \mathrm{o}(1)
\end{equation}
as $k\rightarrow+\infty$. 
To obtain \eqref{normJyk} we have used the fact that $\langle A_T(\chi_\om) z^+_k,y^\infty \rangle = \langle z^+_k,  A_{-T}(\chi_\om) y^\infty \rangle$ converges to to $0$ by weak convergence of $z^+_k$ to $0$ in $L^2$. All other crossed terms converge to $0$ by using a similar argument.

Let $N \in \N^*$. We write $Z_k=Z_k^{\leq N} + Z_k^{>N}$ where $Z_k^{\leq N}$ is the projection on eigenmodes $j \leq N$. 
Since $N$ is fixed, the weak convergence of $Z_k$ to $0$ implies the strong convergence of $Z_k^{\leq N}$ to $0$.
Hence, using the same reasoning as above, we obtain
$$
\Vert Z_k\Vert^2 = \Vert Z_k^{>N} \Vert^2 + \mathrm{o}(1) \quad \hbox{ and } \quad J^{\chi_\omega}_{T}(Z_k) = J^{\chi_\omega}_{T}(Z_k^{>N}) + \mathrm{o}(1)
$$
as $k\to + \infty$. Using \eqref{normyk} and \eqref{normJyk}, we get
$$
\frac{C_T(\om)}{T} = \lim_{k\to +\infty} J^{\chi_\omega}_{T} (Y_k) = \lim_{k\to +\infty} \frac{ J^{\chi_\omega}_{T}(Y_\infty)  + J^{\chi_\omega}_{T}(Z_k^{>N}) + \mathrm{o}(1)}{\Vert Y_\infty\Vert^2 + \Vert Z_k^{>N}\Vert^2 + \mathrm{o}(1)}.
$$
By definition of $C_T^{>N}(\om)$, and $C_T(\om)$, we obtain\footnote{\label{ineqabcd}Here, we use ththe inequality $\frac{a+b}{c+d}\geq \min \left(\frac{a}{c},\frac{b}{d}\right)$ for any positive real numbers $a$, $b$, $c$ and $d$.}
\begin{equation*}
 \begin{split}
  \frac{ J^{\chi_\omega}_{T}(Y_\infty)  + J^{\chi_\omega}_{T}(Z_k^{>N}) + \mathrm{o}(1)}{\Vert Y_\infty\Vert^2 + \Vert Z_k^{>N}\Vert^2 + \mathrm{o}(1)} &  \geq \frac{ \frac{J^{\chi_\omega}_{T}(Y_\infty)}{\Vert Y_\infty\Vert^2} \Vert Y_\infty\Vert^2  + \frac{C_T^{>N}(\om)}{T}  \Vert Z_k^{>N}\Vert^2  + \mathrm{o}(1)}{\Vert Y_\infty\Vert^2 + \Vert Z_k^{>N}\Vert^2 + \mathrm{o}(1)} \\
  & \geq \min\left(  \frac{J^{\chi_\omega}_{T}(Y_\infty)}{\Vert Y_\infty\Vert^2},  \frac{C_T^{>N}(\om)}{T} \right)  + \mathrm{o}(1). 
 \end{split}
\end{equation*}
and therefore $\frac{C_T(\om)}{T}\geq \min\left( \frac{J^{\chi_\omega}_{T}(Y_\infty)}{\Vert Y_\infty\Vert^2},  \frac{C_T^{>N}(\om)}{T} \right)$.
Since $N$ is arbitrary, it follows that
$$
\frac{C_T(\om)}{T}\geq \min\left( \frac{J^{\chi_\omega}_{T}(Y_\infty)}{\Vert Y_\infty\Vert^2},  \alpha^T(\om) \right) .
$$
Since $\frac{C_T(\om)}{T} < \al^T(\om)$ by assumption, we obtain 
$$\frac{C_T(\om)}{T}\geq   \frac{J^{\chi_\omega}_{T}(Y_\infty)}{\Vert Y_\infty\Vert^2}$$
and therefore $\frac{C_T(\om)}{T}$ is reached. 
\end{proof}

\subsection{Large time asymptotics: proof of Theorem \ref{main_long_time}}\label{longtime}
According to Lemma \ref{obvious}, we have $\frac{1}{T}C_T(\omega)\leq \min \left(\frac{1}{2} g_1(\omega),  \alpha^T(\omega) \right)$, and hence 
$$
\limsup_{T\rightarrow +\infty} \frac{C_T(\omega)}{T} \leq \min \left( \frac{1}{2} g_1(\omega),  \alpha^\infty(\omega) \right).
$$ 
Let us prove the converse inequality.  Using the same notations as in the proof of Proposition \ref{compacity}, we consider a sequence $(T_k)_{k\in \N}$ tending to $+\infty$ and 
 $(Y_k)_{k\in\N} = (y^+_k,y^-_k)_{k\in \N} \in (L^2(\Om) \times L^2(\Om))^\N$ a minimizing sequence for $\liminf_{k\to +\infty} \frac{C_{T_k}(\om)}{T_k}$ i.e., a sequence such that  
 \begin{equation} \label{jyk}
  \lim_{k\to +\infty} J^{\chi_\omega}_{T_k} (Y_k) = \liminf_{k\to +\infty} \frac{C_{T_k}(\om)}{T_k}
 \end{equation}
and
 \begin{equation} \label{normyk3}
 \| Y_k \|_{L^2}  =1.
\end{equation}
We write $Y_k = Y_\infty + Z_k$ with $Y_\infty= (y^+_\infty,y^-_\infty)$ and $Z_k=(z^+_k,z^-_k)$ such that $Z_k$ converges weakly to $0$ in $L^2(\Om) \times L^2(\Om)$. Then
\begin{equation} \label{normyk4}
 1 = \Vert Y_k\Vert^2 = \Vert Y_\infty\Vert^2 + \Vert Z_k\Vert^2 + \mathrm{o}(1)
\end{equation}
and 
\begin{equation} \label{normJyk2}
 J^{\omega}_{T_k}(Y_k) = J^{\omega}_{T_k}(Y_\infty)  + J^{\omega}_{T_k}(Z_k)+ \mathrm{o}(1)
\end{equation}
as $k\rightarrow+\infty$. To obtain \eqref{normJyk2} we have used the facts that $\langle A_{T_k}(\chi_\om) z^+_k,y^\infty \rangle = \langle z^+_k,  A_{-T_k}(\chi_\om) y^\infty \rangle$ converges to $0$ by weak convergence of $z^+_k$ to $0$ in $L^2$ and that $A_{-T_k}(\chi_\om)$ converges in $L^2$ to ${A_\infty}(\chi_\om)$ according to Lemma \ref{convAB}. All crossed terms converge to $0$ by using a similar argument.

By Lemma \ref{convAB} (see Section \ref{sec:teclemma}) and by definition of $J^{\omega}_{T_k}$, we get that 
\begin{equation} \label{jinf}
\lim_{k\to +\infty} J^{\omega}_{T_k}(Y_\infty) = \langle \bar A_\infty y^+_\infty, y^+_\infty \rangle +  \langle \bar A_\infty y^-_\infty, y^-_\infty \rangle \geq g_1(\om) \left( \|  y^+_\infty\|_{L^2}^2 + \|  y^-_\infty\|_{L^2}^2 \right) \geq g_1(\om) \|Y_\infty \|^2.
\end{equation}
Writing $z_k=e^{it \Lambda} z^+_k+e^{-it \Lambda} z^-_k$, we have
 $$J^{\omega}_{T_k}(Z_k) = \frac{1}{T_k} \int_0^{T_k} \int_\om \left|z_k\right|^2 \, dv_g \, dt.$$ 
Let $s>0$ and write
$[0,T]=[0,s]\cup[s,2s]\cup\cdots \cup[(m_k-1)s,m_k s] \cup [m_ks,T_k]$ where $m_k$ is the integer part of $T_k/s$. By using several times the inequality of Footnote \ref{ineqabcd}, we obtain
\begin{eqnarray*}
J^{\omega}_{T_k}(Z_k) &=& \frac{ \sum_{j=0}^{m_k-1}\int_{js}^{(j+1)s} \int_\om \left|z_k\right|^2 \, dv_g \, dt+\int_{m_ks}^{T_k} \int_\om \left|z_k\right|^2 \, dv_g \, dt}{T_k}\\
&\geq &  \frac{ \sum_{j=0}^{m_k-1}\int_{js}^{(j+1)s} \int_\om \left|z_k\right|^2 \, dv_g \, dt}{T_k}\\
&=& \frac{ \sum_{j=0}^{m_k-1}\int_{js}^{(j+1)s} \int_\om \left|z_k\right|^2 \, dv_g \, dt}{m_ks}+\left(\frac{1}{T_k}-\frac{1}{m_ks}\right)\int_{0}^{m_ks} \int_\om \left|z_k\right|^2 \, dv_g \, dt \\
&\geq & \min_{1\leq j\leq m_k}\frac{\int_{js}^{(j+1)s} \int_\om \left|z_k\right|^2 \, dv_g \, dt}{s}+\left(\frac{1}{T_k}-\frac{1}{m_ks}\right)\int_{0}^{m_ks} \int_\om \left|z_k\right|^2 \, dv_g \, dt.
\end{eqnarray*}
Using that $0\leq m_ks-T_k<s$, that $T_k\rightarrow + \infty$ and that
$$
\int_{0}^{m_ks} \int_\om \left|z_k\right|^2 \, dv_g \, dt\leq \int_{0}^{m_ks} \int_\Om \left|z_k\right|^2 \, dv_g \, dt=m_ks\Vert Z_k\Vert^2\leq (1+\Vert Y_\infty\Vert^2)m_ks,
$$
we get
$$
J^{\omega}_{T_k}(Z_k)\geq  \min_{1\leq j\leq m_k} J_{\om,s}(\tilde{z}^+_{k,j}, \tilde{z}^-_{k,j})  +   \mathrm{o}(1)\quad \text{with}\quad J_{\om,s}(\tilde{z}^+_{k,j}, \tilde{z}^-_{k,j})= \frac{1}{s} \int_0^{ s} \int_\om \left| \tilde{z}_{k,j}\right|^2 \, dv_g\, dt 
$$
where $(\tilde{z}^+_{k,j}, \tilde{z}^-_{k,j}) $ is the initial condition associated to the solution $z_{k,j}:(t,x)\mapsto z_k(t+js,x)$.

Proceeding as in the proof of Proposition \ref{compacity} and decomposing $\tilde{Z}_{k,j} =(\tilde{z}^+_{k,j}, \tilde{z}^-_{k,j})$ in low/high frequencies as before, we get that, for any nonzero integer $N$, 
 $$
 J_{\om,s}(\tilde{z}^+_{k,j}, \tilde{z}^-_{k,j}) \geq \frac{C_s^N(\om)}{s} \| \tilde{Z}_{k,j} \|^2 + \mathrm{o}(1).
 $$
 Since the wave group is unitary, one has $\| \tilde{Z}_{k,j} \|^2 = \| Z_k \|^2$ and hence 
 $$J_{\om,s}(\tilde{z}^+_{k,j}, \tilde{z}^-_{k,j}) \geq \frac{C_s^N(\om)}{s} \| Z_k \|^2+ \mathrm{o}(1).$$
 Combining these last facts with \eqref{normyk3}, \eqref{normyk4}, \eqref{normJyk2} and \eqref{jinf}, we obtain 
 $$ \liminf_{k\to +\infty} \frac{C_{T_k}(\om)}{T_k} \geq \frac{ g_1(\om) \|Y_\infty \|^2 +  \frac{C_s^N(\om)}{s} \| Z_k \|^2 +  \mathrm{o}(1)}{ \Vert Y_\infty\Vert^2 + \Vert Z_k\Vert^2 + \mathrm{o}(1)} \geq \min\left(g_1(\om),  \frac{C_s^N(\om)}{s} \right) +\mathrm{o}(1).$$
 Since $N$ is arbitrary, we obtain $ \liminf_{k\to +\infty} \frac{C_{T_k}(\om)}{T_k} \geq \min(g_1(\om), \al^s(\om))$, and since $s$ is arbitrary, we conclude that
$$ \liminf_{k\to +\infty} \frac{C_{T_k}(\om)}{T_k} \geq \min(g_1(\om), \al^\infty(\om)) .$$

It remains to show the last claim of the theorem. Let us assume that $g_1(\om) < \alpha^\infty(\om) $. Let us assume by contradiction that $\frac{1}{2}g_1(\om)$ is not reached. Then, there exists a subsequence $(\phi_{j_k})_{k\in\N}$ of eigenfunctions of $-\triangle_g$ normalized in $L^2(\Om)$ such that $j_k\to +\infty$ and $\frac{1}{2}g_1(\om)=\int_\om \phi_{j_k}(x)^2\, dv_g+\operatorname{o}(1)$ as $k\to +\infty$. Now, by definition of $\alpha^\infty(\om)$, by taking $Y_k=(\phi_{j_k},0)$ as initial condition in the infimum defining $C_T^{>N}(\om)$, we infer that $\frac{1}{T}C_T^{>N}(\om)\leq J_T^{\om}(Y_k)$ provided that $k$ be large enough. Passing to the limit with respect to $N$ and $T$ yields $\alpha^\infty(\om)\leq \frac{1}{2}g^1(\om)$, which is a contradiction.

\subsection{Large time asymptotics under the condition $(UG)$: proof of Theorem \ref{spectral_gap1}} \label{proofug}
The proof follows the same lines as the one of Theorem \ref{longtime}. Using the same notations, we have 
\begin{equation*}
  \lim_{k\to +\infty} J^{\chi_\omega}_{T_k} (Y_k) = \liminf_{k\to +\infty} \frac{C_{T_k}(\om)}{T_k}
\end{equation*}
and 
\begin{equation*} 
 1 = \Vert Y_k\Vert^2 = \Vert Y_\infty\Vert^2 + \Vert Z_k\Vert^2 + \mathrm{o}(1),\qquad J^{\omega}_{T_k}(Y_k) = J^{\omega}_{T_k}(Y_\infty)  + J^{\omega}_{T_k}(Z_k)+ \mathrm{o}(1) ,
\end{equation*}
and moreover, 
\begin{equation*} 
\lim_{k\to +\infty} J^{\omega}_{T_k}(Y_\infty)   = \langle \bar A_\infty y^+_\infty, y^+_\infty \rangle +  \langle \bar A_\infty y^-_\infty, y^-_\infty \rangle  \geq g_1(\om) \left( \|  y^+_\infty\|_{L^2}^2 + \|  y^-_\infty\|_{L^2}^2 \right) 
\geq g_1(\om) \|Y_\infty \|^2.
\end{equation*}
Using Lemma \ref{convUG} (see Section \ref{sec:teclemma}), we infer that
\begin{equation*} 
\lim_{k\to +\infty} J^{\omega}_{T_k}(Z_k)   = \langle \bar A_\infty z^+_k, z^+_k\rangle +  \langle \bar A_\infty z^-_k, z^-_k  \rangle \geq g_1(\om) \left( \|  z^+_k \|_{L^2}^2 + \|  z^-_k\|_{L^2}^2 \right) 
\geq g_1(\om) \|Z_k \|^2 ,
\end{equation*}
and thus
$$
 \liminf_{k\to +\infty} \frac{C_{T_k}(\om)}{T_k} \geq \frac{ g_1(\om) \|Y_\infty \|^2 +  g_1(\om)  \| Z_k \|^2 +  \mathrm{o}(1)}{ \Vert Y_\infty\Vert^2 + \Vert Z_k\Vert^2 + \mathrm{o}(1)} .
$$
The conclusion follows.

\subsection{Characterization of observability: proof of Corollary \ref{obser_carac}} \label{observability_cor}
We first observe that $ C_T(\om) > 0$ implies that $ \alpha^T(\om)>0$. Indeed, since $C_T(\om) \leq C_T^{>N}(\om)$ for every $N \in \N^*$, it follows from the definition of $\alpha^T$ that $\alpha^T(\om)=0 \Rightarrow  C_T(\om) =0$. 

Let us prove the converse. Assume by contradiction that 
\begin{equation}\label{eq:alphTCT}
\alpha^T(\om)>0\qquad \textrm{ and }\qquad C_T(\om)=0.
\end{equation}
 For any $s>0$, let us denote by $E_s$ the vector space (sometimes called ``space of invisible solutions'') of initial data $Y=(y^+,y^-)$ in $L^2(\Om) \times L^2(\Om)$ such that $e^{it\Lambda} y^+ e^{-it \Lambda}y^-$ vanishes identically on $[0,s] \times \om$. 

We claim that the following property holds true for every $k \in \N$:
\begin{itemize}
\item[$(H_k)$] For every $\ep>0$ there exists a non trivial $Y_{k,\ep}=(y^+_{k \ep}, y^-_{k \ep}) \in E_{T - \ep}$ involving only frequencies of index greater than $k$, i.e., such that  
$$
\int_{\Omega} y^\pm_{k,\ep}(x) \phi_j(x)\, dv_g(x) = 0,\qquad i=0,1, \quad j=1,\dots,k .
$$
If $k=0$ this property writes: there exists a non trivial solution $Y_{0,\ep} \in E_{T - \ep}$.
\end{itemize}

Admitting this fact temporarily, if $\ep >0$ and $N$ are fixed, Property $(H_{N})$ yields the existence of  $Y_{T,\ep}=(y^+_{T \ep}, y^-_{T \ep}) \in E_{T- \ep}$ involving only frequencies of index higher than $N$ such that $e^{it\Lambda} y^+_{T,\ep} e^{-it \Lambda}y^-_{T,\ep}$ vanishes identically on $[0,T-\ep] \times \om$. Using $Y_{T,\ep}$  as test functions in the functional $J^{\chi_\omega}_{T}$, one infers that 
$C_{T -\ep}^{>N}(\om) = 0$. Note that, without loss of generality, we may assume that $\|Y_{T,\ep}\| = 1$. 
Letting $N$ tend to $+\infty$ yields that $\alpha^{T-\ep}(\om)=0$. Finally,
noting that for all  $(y^+,y^-)$ of norm $1$, one has 
$$
\left| J^{\chi_\om}_{T-\ep}(y^+,y^-) - J^{\chi_\omega}_{T}(y^+,y^-) \right| \leq \frac{\ep}{T-\ep},
$$
we infer that $\alpha^{T} (\om) \leq \alpha^{T- \ep}(\om) +  \frac{\ep}{T-\ep}$ and thus $\alpha^T(\om)=0$, whence the contradiction.  

\medskip

Let us now prove by recurrence that Property $(H_k)$ holds true for every $k \in \N$ under the assumption \eqref{eq:alphTCT}.
Let us first prove that $(H_0)$ is true. According to Theorem \ref{main_finite_time}, the infimum defining $C_T(\om)$ in Definition \eqref{CT} is reached by some $Y=(y^+,y^-)$ such that $e^{it\Lambda} y^+_{T,\ep} e^{-it \Lambda}y^-_{T,\ep}$ vanishes identically on $[0,T] \times \om$. In other words, the dimension of $E_T$ is at least equal to $1$, and this is also true for $E_{T - \ep}$ for any $\ep$ since $E_T \subset E_{T - \ep}$. 

Assume now that $(H_k)$ is true for some $k \in \N$ and let us show that $(H_{k+1})$ is also true. Let $\ep >0$ and let $Y=(y^+,y^-) \in E_{T -{\ep/2}}$ satisfying 
$$
\int_{\Omega} y^\pm(x) \phi_j(x) \, dv_g(x) = 0,\qquad  \textrm{for all }i=0,1, \quad j =1,\dots, k.
$$
Define $y(t,\cdot) = e^{it\Lambda} y^+ e^{-it \Lambda}y^-$.
The crucial point is that for every $s\in [0,\ep/2]$, the function $\tau_s(y): (t,x) \to y(t +s,x)$ belongs to $E_{T-\frac{\ep}{2} - s}$ which is contained in $E_{T- \ep}$.

We now show the existence a $Z=(z^+,z^-)$ such that the function
\begin{eqnarray} \label{z=}
 z:(t,x)\mapsto e^{it\Lambda} z^+(x) e^{-it \Lambda}z^- (x)
\end{eqnarray}
which is a nonzero linear combination of functions $(\tau_s(y))_{s\in [0,\ep/2]}$, satisfies the orthogonality condition
$$
\int_{\Omega} z^\pm(x) \phi_j(x) \, dv_g(x) = 0,\qquad i=0,1,\quad j=1,\dots k+1. 
$$
We expand the solution $\tau_s(y)$ as  
$$
\tau_s(y)(t,\cdot) = \sum_{j=k+1}^{+\infty} \left(a_j(s) e^{i \la_j t} + b_j(s) e^{-i \la_j t} \right) \phi_j(\cdot)
$$ 
where $(a_j(s))_{j\in\N^*}$ and $(b_j(s))_{j\in\N^*}$ belong to $\ell^2(\R)$. In particular, we have
\begin{eqnarray*}
a_j(s) = e^{is \lambda_j } a_j(0) \hbox{ and } b_j(s)= e^{-is \la_j} b_j(0).
\end{eqnarray*}
If $a_{k+1}(0)=b_{k+1}(0)=0$ then $y$ belongs to $E_{T -\ep}$ and involves only frequencies of index higher than $k +1$ which shows that $(H_{k+1})$ holds true. For this reason, we assume that $a_{k+1}(s) \not=0$ or $b_{k+1}(s) \not= 0$. Hence, there exists $j$ such that $\la_j> \la_{k+1}$, and $a_j(0) \not=0$ or $b_j(0) \not=0$. Otherwise, the function $y$ would be a nonzero multiple of an eigenfunction belonging to the eigenspace associated to the eigenvalue $\la_k$ and  would vanish on $\om$: but this is impossible as soon as $\om$ has a positive Lebesgue measure (see \cite{Donnelly_fefferman,Hardt_Simon,lin}), which is the case since $ \alpha^T(\om)>0$. Hence, let us consider $j>k$ such that $\la_j > \la_k$ and $a_j(0) \not=0$ or $b_j(0) \not=0$.
Since $\lambda_j > \la_k$, one can find $0< s<s'\leq  \ep/2$ such that the vectors $(1,e^{i \la_k s}, e^{i \la_k s'})$ and  $(1,e^{i \la_j s}, e^{i \la_j s'})$ are linearly independent. In other words, there exist real numbers $c_0,c_s,c_{s'}$ such that 
\begin{eqnarray} \label{aa'1}
c_0 + c_s e^{i \la_k s} + c_{s'} e^{i \la_k s'} =0
\end{eqnarray}
and
\begin{eqnarray} \label{aa'2}
 c_0 + c_s e^{i \la_j s} + c_{s'} e^{i \la_j s'} \not=0. 
\end{eqnarray}
Then $z= c_0 y+ c_{s} y_{s} +c_{s'} y_{s'}$ is the desired solution. Indeed, writing it as in \eqref{z=}, we obtain $Z \in E_{T-\ep}$ and moreover $z \not=0$ by \eqref{aa'2}. Finally, $z$ involves only frequencies of index larger than $k +1$ by \eqref{aa'1}. This shows $(H_{k+1})$.

\subsection{Convergence properties for $\bar A_T$ and $\bar B_T$}\label{sec:teclemma}
In this section, we establish some convergence properties as $T\rightarrow\infty$ for the operators $\bar A_T(a)$ and $\bar B_T(a)$ introduced in Section \ref{Prelim}. We recall that  $(\la_j)_{j\geq 1}$ denotes the sequence of eigenvalues of $\Lambda=\sqrt{-\triangle_g}$ counted with multiplicity and that $(\phi_j)_{j\geq 1}$ is an orthonormal $L^2$-basis of eigenfunctions of $-\triangle_g$ such that $\phi_j$ is associated to $\la_j^2$.  
Now, let $P_j$ be the $L^2$-projector defined by $P_j y= \langle y, \phi_j \rangle \phi_j$.

Throughout this section, let  $a$ be a bounded nonnegative measurable function, considered as an operator by multiplication.

\begin{lemma} \label{explicit}
We have
$$\bar A_T(a) = \sum_{ j,l \geq 0} f_T(\la_j-\la_l) P_j a P_l \; \hbox{ and } \;  \bar B_T(a)  \sum_{j, l \geq 0} f_T(\la_j-\la_l) P_j a P_l$$
 where $f_T(x) = \left\{\begin{array}{ll} 
              \frac{e^{iTx}-1}{iTx} & \hbox{ if } x \not= 0;\\
              1 &  \hbox{ if } x = 0.
             \end{array}\right.$.
\end{lemma}
 
\begin{proof}
Let $y \in L^2(\Om)$. We set  $y_j= \langle y, \phi_j \rangle$ so that $y= \sum_j y_j \phi_j$. We  have
$$
\bar A_T(\omega) y = \sum_{j} \langle \bar A_T(a) y, \phi_j \rangle \phi_j
= \sum_{j} \left( \sum_{l} \frac{1}{T} \int_0^T e^{it(\la_j-\la_l)}\, dt y_j \int_\Omega a  \phi_j \phi_l \, dv_g\right) \phi_l
$$
and $\frac{1}{T} \int_0^T e^{it(\la_j-\la_l)}\, dt =f_T(\la_j-\la_l)$.
A similar reasoning is done for $\bar B_T(a)$.  
\end{proof}

\begin{lemma} \label{convAB} 
 For every $y=\displaystyle \sum_{j} P_j y = \sum_{j} y_j \phi_j\in L^2(\Omega)$, we have
$$
 \bar A_T(a) y = \frac{1}{T}\int_0^Te^{-it\Lambda} a e^{it\Lambda}\, dt\ y \underset{T\rightarrow \pm \infty}{\longrightarrow} \bar A_\infty(a) y = \sum_{j} \left( y_j \int_\Omega a\phi_j^2\, dv_g\right) \phi_j 
$$
and
$$
 \bar B_T(a) y = \frac{1}{T}\int_0^Te^{-it\Lambda} a e^{it\Lambda}\, dt\ y \underset{T\rightarrow \pm \infty}{\longrightarrow} 0.
$$
In other words, the operator $\bar A_T(a)$ (resp. $\bar B_T(a)$)  converges pointwisely to a diagonal operator (resp. $0$) in $L^2(\Omega)$ as $T\rightarrow\pm\infty$.
\end{lemma}

\begin{proof}
 Let $l$ be a fixed integer. We first show that 
 \begin{eqnarray} \label{wc}
  \lim_{T \to \pm \infty} \langle A_T(a) y , \phi_l \rangle =  \langle A_\infty(a) y , \phi_l \rangle
 \end{eqnarray}
Let $N \in \N$. Setting 
 $$r_N =  \sum_{j>N} \frac{y_j}{T} \int_0^T e^{it(\la_j-\la_l)}\, dt \int_\Omega a  \phi_j\phi_l dv_g \ \in\C,$$ 
 we have
$$
\langle \bar A_T(a) y, \phi_l\rangle = \sum_{j \leq N} f_T(\la_j - \la_l)  y_j \int_\Omega a(x)  \phi_j\phi_l dv_g(x) + r_N.
$$
If $\la_j \neq \la_l$ then $f_T(\la_j - \la_l)  \rightarrow 0$ as $T\rightarrow \pm\infty$, and if $\la_j=\la_l$ then $f_T(\la_j - \la_l) =1$. Therefore the limit of the finite sum above is equal to $y_l \int_\Omega a(x) \phi_l^2 dv_g(x)$. Let us prove that $r_N$ is arbitrarily small if $N$ is large enough.
Setting $y^N = \sum_{j>N} y_j \phi_j$ (high-frequency truncature) and considering $C>0$ such that $a\leq C$ a.e. in $\Om$, we have
\begin{equation*}
\begin{split}
\vert r_N\vert  &= \left\vert \frac{1}{T} \int_0^T \int_\Omega \sum_{j>N} e^{it\la_j} y_j \phi_j(x) e^{-it\la_l} \phi_l(x)\, dv_g(x)\, dt \right\vert = \left\vert \frac{1}{T} \int_0^T \int_\Omega a(x) (e^{it\Lambda}y^N)(x) e^{-itl} \phi_l(x)\, dv_g(x)\, dt \right\vert \\
&\leq \frac{C}{T} \int_0^T \int_\Omega \vert (e^{it\Lambda}y^N)(x)\vert  \vert \phi_l(x)\vert \, dv_g(x)\, dt \leq \left( \frac{1}{T} \int_0^T \Vert e^{it\Lambda}y^N\Vert^2_{L^2} \, dt \right)^{1/2} = \Vert y^N\Vert^2_{L^2}
\end{split}
\end{equation*}
since $e^{it\Lambda}$ is an isometry in $L^2(\Omega)$. 
Therefore $r_N=\operatorname{o}(1)$ as $N\to +\infty$.

We have proved that $\langle \bar A_T(a) y, \phi_l \rangle \rightarrow y_l \int_\Omega \phi_l^2 dv_g(x)$ as $T\rightarrow \pm \infty$ and then \eqref{wc} is true. It follows that  $\bar A_T(a) y \rightharpoonup \bar A_\infty(a) y$ for the weak topology of $L^2(\Om)$. 

Let us now write $y=y_N+y^N$ with $y_N = \sum_{j\leq N} y_j \phi_j$ and $y^N = \sum_{j>N} y_j \phi_j$. By compactness for frequencies lower than or equal to $N$, we have $\bar A_T(a) y_N \rightarrow \bar A_\infty(a) y_N$ for the strong topology of $L^2(\Omega)$. Besides, noting that $\Vert \bar A_T(a)\Vert\leq 1$, we have $\Vert \bar A_T(a)y^N\Vert\leq \Vert y^N\Vert$, and since $\Vert y^N\Vert$ can be made arbitrarily small by taking $N$ large, the result follows.

The same argument allows to prove that $\bar B_T(a) y$ tends to $0$ when $T \to \pm \infty$.  
\end{proof}

\begin{lemma} \label{convUG}
Under $(UG)$, $\bar A_T(a)$ converges uniformly (i.e., in operator norm) to $\bar A_\infty(a)$ as $T\rightarrow\pm\infty$.
\end{lemma}

\begin{proof}
%
%
It suffices to prove that
$$
\lim_{T\rightarrow +\infty} \sup_{ \sum_j |y_j|^2=\sum_l |z_l|^2=1} \sum_{j\neq\l} f_T(\la_j-\la_l) \langle a\phi_j,\phi_l\rangle y_l z_l  = 0.
$$

Since $\vert f_T(\la_j-\la_l)\vert\leq \frac{2}{T\vert\la_j-\la_l\vert}$, we have
$$
\left\vert \sum_{j\neq\l} f_T(\lambda_j-\lambda_j) \langle a\phi_j,\phi_l\rangle y_j z_l \right| \vert
\leq \frac{2}{T} \sum_{j\neq\l} \frac{\vert y_j\vert \vert z_l \vert}{\vert\la_j-\la_l\vert} \leq \frac{\mathrm{C}}{T} ,
$$
as a consequence of Montgomery-Vaughan's inequality (recalled below) and where $C>0$ is independent of $(y_j)_{j\in \N}$, $(z_l)_{l\in \N}$, $(\phi_j)_{j\in \N}$, $(\phi_l)_{l\in \N}$.  The result follows.
\end{proof}

The well known Hilbert inequality states that
$$
\left\vert \sum_{j\neq k} \frac{a_jÊ\bar b_k}{j-k} \right\vert^2  \leq \pi^2 \sum_{j=1}^{+\infty} \vert a_j\vert^2 \sum_{j=1}^{+\infty} \vert b_j\vert^2 \qquad \forall (a_j)_{j\in\N},(b_j)_{j\in\N}\in\ell^2(\C).
$$
The same statement holds true with $j-k$ replaced with $j+k$.
A generalization by Montgomery and Vaughan in \cite{MV} states that, given $\lambda_1 < \cdots < \lambda_j < \cdots$ with $\lambda_{j+1}-\lambda_j\geq \delta >0$ for every $j$ (uniform gap), one has
$$
\left\vert \sum_{j\neq k} \frac{a_jÊ\bar b_k}{\lambda_j-\lambda_k} \right\vert^2  \leq \frac{\pi^2}{\delta^2} \sum_{j=1}^{+\infty} \vert a_j\vert^2 \sum_{j=1}^{+\infty} \vert b_j\vert^2 \qquad \forall (a_j)_{j\in\N},(b_j)_{j\in\N}\in\ell^2(\C).
$$
%

\section{Concluding remarks and perspectives}\label{sec:conclusion}

We provide here a list of open problems and issues.

\paragraph{Manifolds with boundary.} 
The introduction of the so-called {\it high-frequency observability constant $\alpha^T(\om)$} is of interest because of the equivalence $C_T(\om)>0\Leftrightarrow\alpha^T(\om)>0$ stated in Corollary \ref{obser_carac}. It is still true on a manifold with boundary. But then extending Theorem \ref{computation} and Corollary \ref{obser_carac_explicit} to manifolds with boundary raises difficulties.

\paragraph{Schr\"odinger equation.} 
It is known that GCC implies internal observability of the Schr\"odinger equation (see \cite{Lebeau_JMPA1992}), but this sufficient condition is not sharp (see \cite{Jaffard}). Until now a necessary and sufficient condition for observability is still not known (see \cite{laurent}). We think that some of the approaches developed in this paper, combined with microlocal issues, may serve to address this problem.

\paragraph{Shape optimization.} 
A challenging problem is to maximize the functional $\om \mapsto C_T(\om)$ over the set of all possible measurable subsets of $\Omega$ of measure $|\omega|=L|\Omega|$ for some fixed $L\in (0,1)$. In \cite{PTZObs1,PTZobsND}, the maximization of the \emph{randomized} observability constant has been considered, that is, the functional $\om\mapsto g_1(om)$. Maximizing the functional $\om\mapsto g_2(\om)$ is an interesting open problem which, thanks to Corollary \ref{cormain}, would be a step towards the maximization of the \emph{deterministic} observability constant.

\end{document}